\documentclass[a4paper,12pt,reqno]{article}
\usepackage[T1]{fontenc}
\usepackage{authblk}
\usepackage{epsfig,graphicx,subfigure}
\usepackage{amsthm,amsmath,amssymb,amsfonts}
\usepackage{color,xcolor}
\usepackage[all]{xy}
\usepackage[top=20mm, bottom=10mm, left=15mm, right=15mm]{geometry}
\usepackage{hyperref}%The package hyperref provides LaTeX the ability to create hyperlinks within the document

\newtheorem{theorem}{Theorem}[section]
\newtheorem{lemma}[theorem]{Lemma}
\newtheorem{proposition}[theorem]{Proposition}

\theoremstyle{definition}
\newtheorem{definition}[theorem]{Definition}
\newtheorem{example}[theorem]{Example}

\newtheorem{remark}[theorem]{Remark}
\numberwithin{equation}{section}
\begin{document}
\setcounter{page}{1}
\title{\vspace{-1.5cm}
{\large\textit {}
\\[+.1cm]{\large{\bf A representation type of compression space of rank 2 } }}}
\date{}
\author{{\large \vspace{-2mm}  }Hossein Kheiri}
%\affil{\large{\vspace{-3mm} }Department of Mathematical Sciences, Isfahan University of Technology, Isfahan 84156-83111, Iran}
%\affil{{\large \vspace{-4mm}}Department of Mathematical Sciences, Amirkabir University of Technology, Tehran 15916-34311, Iran}
\affil{\large{\vspace{-3mm}  }h.kheiri@math.iut.ac.ir}
%\affil{\large \vspace{-4mm}The Name of University, Email of the third author}

%\renewcommand\Authands{ and }
%\baselineskip=.8cm
\maketitle
\vspace{-05mm}
\begin{abstract}
\normalsize % Add normalsize right at the begining of the abstract
In this paper, we proved that a compression space of rank $2$ is equivalent to an irreducible representation over a Lie algebra.
\end{abstract}
%\bigskip
\noindent \textbf{keywords}: Linear space, Compression space, Line bundles, Irreducible representations.
%\vskip 0.4cm

%\bigskip

\baselineskip=.8cm
\maketitle

\section{Introduction}
Let $V$ and $W$ be two complex vector spaces, and consider a vector subspace $$M \subset Hom(V,W) \simeq V^{\vee} \otimes W.$$ $M$ is a space of matrices of constant rank $k$ if all its non-zero elements have the rank $=$ $k$.
In this paper, we apply the definition of compression space that was given in \cite{EH}. $M$ is called a compression space if there exist subspaces $V^\prime \subset V$ of codimension $k_{1}$ and $W^\prime \subset W$ of dimension $k_{2}$
such that
\begin{enumerate}
  \item $k = k_{1}+k_{2}$, and
  \item every element of $M$ maps $V^\prime$ into $W^\prime$.
\end{enumerate}
By the above assumption, if the rank of $M$ is one, hence $M$ is a compression space. But, it does not necessarily hold for every space of higher rank.
For example, the space of $3 \times 3$ skew-symmetric matrices has rank $2$, however it is not a compression space.
A compression space of rank $k$ is equivalent to a
space of dim$V$ $\times$ dim$W$ matrices having a common $u_{1}$ $\times$ $w_{1}$ block of
zeros with
\begin{center}
  dim $V$ - $u_{1}$ + dim $W$ - $w_{1}$ = $k$,
\end{center}
the largest possible value. For example, if
codim $V^\prime$ = dim $W^\prime$ = 1, then a compression space of rank $2$ is equivalent to the space of matrices of the form

\[
\begin{pmatrix}
0 & 0 & \cdots & 0 & * \\
0 & 0 & \cdots & \cdots & * \\
\vdots  & \vdots  & \ddots & \vdots  \\
0 & 0 & \cdots & 0 & * \\
* & * & \cdots & * & *
\end{pmatrix}.
\]
This note has expanded from our attempt \cite{KH} to construct a general theory of representations over the space of global sections of an algebraic vector bundles and can be regarded as a first step towards its systematic exposition. We decided to present it as a separate work. So we want to devote a question of linear algebra motivated by the study of vector bundles over the projective space. More precisely, the aim of this note is a classification of a compression space of rank $2$ by using representations over the space of global sections of a vector bundle that we will show in section \ref{a}. However, more research is needed before being able to associate between compression spaces of higher rank and
representations over Lie algebras.
\section{Representation and Vector Bundle}
According to Algebraic Geometry written by Hartshorne \cite{R.H}, any geometrical vector bundle $E$ of rank $m$ over $X$ is correspondence to a locally free sheaf of rank $m$ over $X$.
If vector bundle $E$ is trivial, then the
corresponding locally free sheaf is isomorphic to $m$-sums ${O_{X} \bigoplus \cdots \bigoplus O_{X}}$ where $O_{X}$ is the structure sheaf over $X$. Conversely is true.
Therefore, if there is no problem, according to the above hypotheses, consider these two concepts instead of each other.

Through of this paper, we work over complex vector bundles of finite rank $m$ and $\Gamma(X,E)$ is the space of global sections of a vector bundle $E$ over $X$ where $X$ is a finite dimensional projective space over $\mathbb{C}.$ Any unexplained terminology and all the basic result on representation theory that are used in the sequel can be find in \cite{W.F}.
\begin{definition}
\label{7}
A basis of sections for a vector bundle $E$ over $X$ is a collection of sections $e_{1},\cdots , e_{n}$ such that the vectors $e_{1}(x),\cdots,e_{n}(x)$ are linearly independent in each fiber $E_{x}$ for $x \in X.$ The generated vector space by a basis of sections is denoted by $V_{E}.$
\end{definition}
\begin{definition}
\label{8}
Let $E$ be a vector bundle over $X$ and a linear subspace $\{0\} \neq Z \subset \Gamma(X,E)$ where $\Gamma(X,E)$ is the space of global sections of $E$ over $X$. Define the map
\begin{center}
 $ \mu_{Z} : X \times Z \to E$
\end{center}
given by $\mu_{Z} (x,s) = s(x)$ for every $(x,s) \in X \times Z.$ We said $E$ is generated by the sections of $Z$ if $\mu_{Z}$ is surjective.\\
\end{definition}
\begin{example}
Let $g$ be a generated vector space by three sections $e_{1}, e_{2}$ and $e_{3}$.
Consider the flowing Lie bracket
\begin{center}
  $[e_{3}, e_{1}]:= e_{1} , [e_{3}, e_{2} ] := e_{2} , [e_{1} , e_{2}] := 0$
\end{center}
over $g$. Then $g$ is a solvable Lie Algebra. For more about construction of three dimensional solvable Lie algebra see \cite{W.A}.
\end{example}
\begin{example}
Let $T_{n}(k)$ be the Lie algebra of all $n \times n$ upper triangular matrices over a field $k$  with the Lie product $[ A,B] := AB-BA$. $T_{n}(k)$ is a solvable Lie algebra. Let $O(2)$ be a line bundle over $\mathbb{P} ^{1}$. dim $\Gamma (\mathbb{P} ^{1}, O(2)) = $ dim $T_{2}(\mathbb{C})=3.$ Two finite dimensional vector spaces are isomorphic if and only if they have the same dimension. So there is a bijective map of vector spaces
\begin{center}
  $\varphi : \Gamma (\mathbb{P} ^{1}, O(2)) \to T_{2}(\mathbb{C})$
\end{center}
now define the Lie bracket
\begin{center}
  $[ x , y ] : = [ \varphi (x) , \varphi (y) ] = \varphi (x) \varphi (y) - \varphi (y) \varphi (x)$
\end{center}
over $\Gamma (\mathbb{P} ^{1}, O(2))$. Since $T_{n}(\mathbb{C})$ is a solvable Lie algebra then we can consider $\Gamma (\mathbb{P} ^{1}, O(2))$ as a solvable Lie algebra. Hence any generated vector space by a basis of sections of $ O(2)$ can be consider as a solvable Lie algebra.
\end{example}
\begin{remark}
  From now on, Under the assumption of Definitions \ref{7} and \ref{8}, we will assume $ V_{E}$ is a solvable Lie algebra.
\end{remark}
\begin{remark}
By corresponding between vector bundles and locally free sheaves, a vector bundle $E$ of finite rank $m$ is called trivial if the corresponding locally sheaf of $E$ is isomorphic to the $m$-sum ${O_{X} \bigoplus \cdots \bigoplus O_{X}}$.
\end{remark}
\begin{proposition}
(Ado theorem, [9, Theorem E.4]) Every Lie algebra has a faithful finite dimensional representation.
\end{proposition}
\begin{lemma}
\label{1}
Let $E$ be a vector bundle over $X$ and $E$ be generated by the sections of $V_{E}$ where $V_{E}$ is a solvable Lie algebra. Then $E$ is isomorphic to $O_{X}$ if and only if there exists a pair $(V_{E},\pi)$ where $\pi$ is an irreducible representation over $V_{E}.$
\end{lemma}
\begin{proof}
($\Longrightarrow$) Suppose that $E$ is a trivial line bundle. So there is a basis space $V_{E}$ for the vector bundle such that
dim $V_{E} =1.$  Therefore, there is an injective irreducible representation over $V_{E}$.

($\Longleftarrow$) Conversely, Consider that a vector bundle of rank $ m $, $E$ is isomorphic to the trivial bundle if and only if
it has $m$ sections $s_{1}, \cdots ,s_{m}$ such that the vectors $s_{1} (x),\cdots,s_{m} (x)$ are linearly independent in each fiber $E_{x}$ \footnote{For more details we refer the reader to \cite{HA}}.
If $E$ is a nontrivial line bundle, then
\begin{center}
  dim $V_{E} \neq 1$
\end{center}
for any basis space $V_{E}$ of $E$. By assumption, there exists an irreducible representation $$ \pi : V_{E} \to gl(V)$$ over $V_{E}$ where $ V $ is a finite dimensional complex vector space. Thus, there are irreducible representations $ \pi _{j} $ such that
\begin{center}
 $\pi = \bigoplus \pi _{j}$
\end{center}
 where $ \pi _{j} : V_{E} \to gl(V_{j})$ and $ V_{j} $ is a subspace of $ V $. Since $V_{E}$ is solvable and $ \pi _{j} $ is irreducible, we have
 \begin{center}
   dim $V_{j} = 1.$
 \end{center}
By applying the Ado theorem, we can consider $ \pi $ as an injective representation over $V_{E}$. So dim $V_{E} \leq 1$.

If  dim $V_{E} = 1$,  then the vector bundle is a trivial line bundle.
Consequently, $E$ is a line bundle.
Finally, if we apply the Ado's theorem again, we can delete the injective assumption so the result is obtained.
\end{proof}
As an application of the previous lemma, we have the following lemma.
\begin{lemma}
\label{2}
Let $E $ be a line bundle over $X$ and $E$ be generated by the sections of $V_{E}$, then $E $ is isomorphic to $O_{X}(n)$ for $n \neq 1$ if and only if there is not a pair $(V_{E},\pi)$ where $\pi$ is an irreducible representation over $V_{E}.$
\end{lemma}
\begin{proof}
It is a direct result of Lemma \ref{1}.
\end{proof}
\begin{theorem}
  There is a correspondence between decomposition of a vector bundle as a direct sum of line bundles and direct sum of representations over subalgebras.
\end{theorem}
\begin{proof}
By applying Lemma \ref{1} and Lemma \ref{2}, if $V_{L} $ is a solvable Lie algebra for line bundle $ L $, then we have a one-to-one correspondence as follows:
\begin{equation*}
\left \{ \begin{matrix}
\mbox{trivial }\\
\mbox{(nontrivial)} \\
\mbox{complex line bundle}\\
L
\end{matrix} \right\}
\longleftrightarrow
\left \{ \begin{matrix}
\mbox{ irreducible}\\
\mbox{(reducible)}\\
\mbox{representation over}\\
V_{L}
\end{matrix} \right\}
\end{equation*}
So if $ E $ is a vector bundle over $X$, i.e. it is isomorphic to direct sum of line bundles $ E_{i}, $ then
\begin{align*}
 \Gamma (X, E) \simeq \oplus \Gamma (X, E _{i}).
\end{align*}
Since $V_{E} $ is a subalgebra of $\Gamma (X, E)$ hence
\begin{align*}
 V_{E} \simeq \oplus V_{E_{i}}.
\end{align*}
If $ \pi : \oplus V_{E_{i}} \to gl(V) $ is a representation, then there are $ V_{i} \subset V $ such that $ \pi = \oplus \pi _{i} $ where $ \pi _{i} : V_{E_{i}} \to gl(V_{i}). $ Therefore,
we have:
\begin{enumerate}
\item if $ \pi _{i} $ is an irreducible representation for every $ i $, then $ E$ is a direct sum of trivial line bundles.
\item If $ \pi _{i} $ is a reducible representation for every $ i,$ then $ E $ is a direct sum of nontrivial line bundles.
\item If $ \pi _{j} $ is an irreducible representation for $ {j} \in {J} \subseteq {I} $ and $ \pi _{k} $ is a reducible representation for $ {k} \in {K} \subseteq {I} $ then the vector bundle is isomorphic to direct sum of some trivial and nontrivial line bundles.
\end{enumerate}
Therefore, there is a correspondence between direct sum of complex line bundles over $ X $ and direct sum of representations over Lie algebras, i.e.
\begin{equation*}
\left \{ \begin{matrix}
\mbox{direct sum of line bundls}\\
\mbox{ over X}
\end{matrix} \right\}
\longleftrightarrow
\left \{ \begin{matrix}
\mbox{direct sum of representations }\\
\mbox{over subalgebra of $\Gamma (X,-)$}
\end{matrix} \right\}
\end{equation*}
\end{proof}
\section{Compression spaces and Representation }
\label{a}
Through of this section, $M$ is a compression space and $X$ is the complex projective space. Let $P = \mathbb{P}M$ be the projective space of one dimensional subspaces of $M$. There is a map of vector bundles
\begin{center}
  $\phi : V \otimes O_{P}(-1) \to W \otimes O_{P}$
\end{center}
sending a vector $v \otimes \lambda A$ to the vector $\lambda. A(v)$ for every $A \in M.$ If we twist $\phi$ by $ O_{P} (1)$, we obtain the following maps
\begin{center}
 $\phi_{M} : V \otimes O_{P} \to W \otimes O_{P} (1)$
\end{center}
and
\begin{center}
  $ \phi _{M^*} : W^{*} \otimes  O_{P} \to V^{*} \otimes O_{P} (1).$
\end{center}
It is easy to see that Im($\phi _{M}$) and Im($\phi _{M^*}$) are vector bundles of the rank $M$.
Let
Im($\phi _{M}$) and Im($\phi _{M^*}$) have as direct summands vector bundles $\bigoplus$ $L_{i}$ and $\bigoplus$ $T_{j}$ respectively in which $L_{i}$ and $T_{j}$ are vector bundles. For more details, see \cite{EH}.
\begin{proposition}
According to the above assumptions, the following conditions are equivalent:

\begin{enumerate}
  \item $M$ is a compression space of rank $2$.
  \item At least there are two vector bundles $L_{t}$ and $T_{s}$ such that $\mu_{V_{L_{t}}}$ and $\mu_{V_{T_{s}}}$ are surjective and there are irreducible representations over $V_{L_{t}}$ and $V_{T_{s}}$.
\end{enumerate}

\end{proposition}
\begin{proof}
$(1) \Longrightarrow (2)$
If $M$ is a compression space of rank $2$, by applying the Proposition 2.1 of \cite{EH}, that is, it is equivalent to $L_{i}$'s and $T_{j}$'s are trivial vector bundles of rank $k_{1}$ and rank $k_{2}$ such that
\begin{center}
  rank $M = k_{1} + k_{2}.$
\end{center}
So $k_{1} + k_{2} = 2$ then $k_{1} = k_{2} = 1.$\\
Hence $L_{i}$'s and $T_{j}$'s are trivial line bundles. By applying Lemma \ref{1}, so there are irreducible representations over $V_{L_{t}}$ and $V_{T_{s}}$ for some $t,s$.

$(2) \Longrightarrow (1)$
Let $\pi$ and $\pi ^\prime$ be irreducible representations over $V_{L_{t}}$ and $V_{T_{s}}$ for some $t,s$ respectively. Since
$V_{L_{t}}$ and $V_{T_{s}}$ are solvable then dimensions of $V_{L_{t}}$ and $V_{T_{s}}$ are equal one.
By Lemma \ref{1}, so $L_{t}$ and $T_{s}$ are trivial line bundles.
Therefore,
\begin{center}
 dim $V_{L_{t}}$ + dim $V_{T_{s}}$ = rank $M$
\end{center}
so the rank of $M$ is 2.\\
It is a direct result from Proposition 2.1 of \cite{EH} that $M$ is a compression space.
\end{proof}
Studying a space of matrices of constant rank k is an interesting topic in algebraic geometry.
J. Sylvester \cite{JS} used algebraic geometry as a tool to estimate the maximal
dimension of a linear subspace $M$ when $V$ and $W$ are complex vector spaces.
These subspaces are used in the study of vector bundles. In \cite{AD}, the authors constructed non-splitting vector bundles on the space $\mathbb{P}(S^{d} \mathbb{C} ^{n+1} )$ of symmetric
forms of degree $d$ in $n+1$ variables such that they are stable according to Mumford-Takemoto or slope-stable. They used the representation theory of $SL_{n+1} (\mathbb{C})$ to achieve their result. For more applications in the vector bundles, refer to \cite{EP}, \cite{BF}, \cite{KH} and \cite{BD}. For other applications, see for instance, \cite{DH}, \cite{BM}, \cite{KB} and \cite{BI}.
D. Eisenbud and J. Harris \cite{EH} proved a classification of a certain torsion-free sheaves on the projective space by investigating compression and primitive spaces.
As an interesting problem, it can be studied the relation between compression spaces of higher ranks and other representations over a Lie algebra
. So it can be used as a tool for the study of the classifications of vector bundles, for example, the first chern class of a vector bundle and
representations of the global sections may be considered as a useful tool.

\bibliographystyle{amsplain}

\end{document}